\documentclass[reqno,a4paper,12pt]{amsart}

\usepackage[all,poly]{xy}
\usepackage{amsfonts,stmaryrd}
\usepackage[mathcal]{eucal}
\usepackage{amssymb}
\usepackage{amsmath}
\usepackage{mathrsfs}
\usepackage{color}
\usepackage[pagebackref,colorlinks]{hyperref}
\usepackage{enumerate}

\theoremstyle{plain}
\newtheorem {Lem}{Lemma}
\newtheorem {The}{Theorem}
\newtheorem {Cor}{Corollary}

\newtheorem {Prob}{Problem}

\theoremstyle{remark}

\theoremstyle{definition}

\newcommand{\GL}{\operatorname{GL}}
\newcommand{\SL}{\operatorname{SL}}
\newcommand{\Sp}{\operatorname{Sp}}
\newcommand{\map}{\longrightarrow}
\newif\ifcomm
\let\ifcomm\iffalse

\newcommand\supb[1]{{}^{\fboxsep1pt\fbox{$\scriptstyle\mskip1mu#1\mskip-1mu$}}}

\def\a{\alpha}
\def\b{\beta}
\def\g{\gamma}

\def\A{\operatorname{A}}
\def\B{\operatorname{B}}
\def\C{\operatorname{C}}

\def\F{\operatorname{F}}
\def\G{\operatorname{G}}

\def\K{\operatorname{K}}

\def\rk{\operatorname{rk}}
\def\sr{\operatorname{sr}}
\def\dim{\operatorname{dim}}
\def\Max{\operatorname{Max}}

\title[Generation of relative commutator subgroups]{Generation of relative commutator subgroups\\ in Chevalley groups}

\author{R.~Hazrat}
\address{
Centre for Research in Mathematics\\
University of Western Sydney\\
Australia}
\email{r.hazrat@uws.edu.au}

\author{N.~Vavilov}
\address{Department of Mathematics and Mechanics\\
St.~Petersburg State University\\ St.~Petersburg, Russia}
\email{nikolai-vavilov@yandex.ru}

\author{Z.~Zhang}
\address{Department of  Mathematics\\
 Beijing Institute of Technology\\
 Beijing, China}
\email{zuhong@gmail.com}

\begin{document}

\begin{abstract}
Let $\Phi$ be a reduced irreducible root system of rank $\ge 2$,
let $R$ be a commutative ring and let $I,J$ be two ideals of $R$.
In the present paper we describe generators of the commutator
groups of relative elementary subgroups $\big[E(\Phi,R,I),E(\Phi,R,J)\big]$
both as normal subgroups of the elementary Chevalley group $E(\Phi,R)$,
and as groups. Namely, let $x_{\a}(\xi)$, $\a\in\Phi$, $\xi\in R$,
be an elementary generator of $E(\Phi,R)$. As a normal subgroup
of the absolute elementary group $E(\Phi,R)$, the relative elementary
subgroup is generated by $x_{\a}(\xi)$, $\a\in\Phi$, $\xi\in I$.
Classical results due to Michael Stein, Jacques Tits and Leonid
Vaserstein assert that {\it as a group\/} $E(\Phi,R,I)$ is generated
by $z_{\a}(\xi,\eta)$, where $\a\in\Phi$, $\xi\in I$, $\eta\in R$.
In the present paper, we prove the following {\it birelative\/} analogues
of these results. As a normal subgroup of $E(\Phi,R)$ the relative
commutator subgroup $\big[E(\Phi,R,I),E(\Phi,R,J)\big]$ is generated
by the following three types of generators:
i) $\big[x_{\alpha}(\xi),z_{\alpha}(\zeta,\eta)\big]$,
ii) $\big[x_{\alpha}(\xi),x_{-\alpha}(\zeta)\big]$, and iii)
$x_{\alpha}(\xi\zeta)$, where $\alpha\in\Phi$, $\xi\in I$,
$\zeta\in J$, $\eta\in R$. As a group, the generators are essentially
the same, only that type iii) should be enlarged to
iv) $z_{\alpha}(\xi\zeta,\eta)$. For classical groups, these results,
with much more computational proofs, were established in previous
papers by the authors. There is already an amazing application of
these results, namely in the recent work of Alexei Stepanov~\cite{stepanov10} on
relative commutator width.
\par\smallskip\noindent

\end{abstract}

\maketitle


\section*{Introduction}

Let $\Phi$ be a reduced irreducible root system of rank $\ge 2$,
let $R$ be a commutative ring with 1, and let $G(\Phi,R)$ be a
Chevalley group of type $\Phi$ over $R$. For the background on
Chevalley groups over rings see \cite{NV91} or \cite{vavplot},
where one can find many further references. We fix a split maximal
torus $T(\Phi,R)$ in $G(\Phi,R)$ and consider root unipotents
$x_{\alpha}(\xi)$
elementary with respect to $T(\Phi,R)$. The subgroup $E(\Phi,R)$
generated by all $x_{\alpha}(\xi)$, where $\alpha\in\Phi$,
$\xi\in R$, is called the [absolute] elementary subgroup of
$G(\Phi,R)$.
\par
Now, let $I\unlhd R$ be an ideal of $R$. Then the relative
elementary subgroup $E(\Phi,R,I)$ is defined as the {\it normal\/}
subgroup of $E(\Phi,R)$, generated by all elementary root unipotents
$x_{\alpha}(\xi)$ of level $I$,
$$ E(\Phi,R,I)=
{\big\langle x_{\alpha}(\xi)\mid \alpha\in\Phi,\
\xi\in R\big\rangle}^{E(\Phi,R)}. $$
\noindent
In other words, as a normal subgroup of $E(\Phi,R)$, the relative
elementary subgroup $E(\Phi,R,I)$ is generated by $x_{\alpha}(\xi)$,
where $\alpha\in\Phi$, $\xi\in I$.
\par
A starting point of the present work is the following classical
result. Morally, it goes back to the Thesis of Michael Stein
\cite{stein2}, which contains all calculations necessary
for its proof. In a slightly weaker form it was first stated
by Jacques Tits \cite{tits}. Namely, Proposition 1 there
asserts that $E(\Phi,R,I)$ is generated by its intersections with
the fundamental $\SL_2$'s. However, the earliest reference for
the precise form below, we could trace, is the work by Leonid
Vaserstein \cite{vaser86}, Theorem~2.

\begin{The}[{Stein---Tits---Vaserstein}]\label{The:1}
Let $\Phi$ be a reduced irreducible root system of rank $\ge 2$
and let $I$ be an ideal of a commutative ring $R$. Then
as a group\/ $E(\Phi,R,I)$ is generated by the elements of
the form
$$ z_{\alpha}(\xi,\eta)=
x_{-\alpha}(\eta)x_{\alpha}(\xi)x_{-\alpha}(-\eta), $$
\noindent
where\/ $\xi\in I$ for $\alpha\in\Phi$, while\/ $\eta\in R$.
\end{The}

In the present paper, we delve a bit deeper into the structure of
relative $K_1$'s, proving a {\it birelative\/} version of the
above result. Namely, we determine elementary generators of
the mixed commutator subgroups $[E(\Phi,R,I),E(\Phi,R,J)]$
of two relative elementary subgroups $E(\Phi,R,I)$ and
$E(\Phi,R,J)$, corresponding to ideals $I,J\unlhd R$.
\par
It is easy to prove that, apart from the known small exceptions,
one has
$$ E(\Phi,R,IJ)\le\big [E(\Phi,R,I),E(\Phi,R,J)\big]\le G(\Phi,R,IJ), $$
\noindent
where $G(\Phi,R,I)$ denotes the principal congruence subgroup
of $G(\Phi,R)$, of level $I$, see Theorem 4 below for a precise
statement.
\par
Thus, in particular, $[E(\Phi,R,I),E(\Phi,R,J)]$ contains all
elements $z_{\alpha}(\xi\zeta,\eta)$, where $\alpha\in\Phi$,
$\xi\in I$, $\zeta\in J$, $\eta\in R$. In some cases, for
instance when $I$ and $J$ are comaximal, $I+J=R$, one can
prove that
$$ [E(\Phi,R,I),E(\Phi,R,J)]=E(\Phi,R,IJ). $$
\noindent
However, easy examples show that in general the mixed commutator
subgroup can be strictly larger than $E(\Phi,R,IJ)$ even for
very nice 1-dimensional rings, see \cite{MAS3,Mason74,MAS1}. Since
$E(\Phi,R,IJ)$ is already normal in $E(\Phi,R)$, one needs
{\it additional\/} generators to span $[E(\Phi,R,I),E(\Phi,R,J)]$
as a subgroup, and even as a normal subgroup of $E(\Phi,R)$.
\par
In the present paper we display the missing generators. As in
the case of the relative elementary subgroups themselves, these
generators sit in the fundamental $\SL_2$'s and are in fact
commutators of {\it some\/} elementary generators of $E(\Phi,R,I)$
and $E(\Phi,R,J)$.
\par
Now we are in a position to state the main results of the present
paper, which provide such sets of generators of
$[E(\Phi,R,I),E(\Phi,R,J)]$ as a normal subgroup of $E(\Phi,R)$,
and as a group, respectively.

\begin{The}\label{The:2}
Let $\Phi$ be a reduced irreducible root system of rank $\ge 2$ and $I$, $J$ be two ideals of a commutative ring $R$.
In the cases\/ $\Phi=\C_2,\G_2$ assume that $R$ does not have residue
fields\/ ${\Bbb F}_{\!2}$ of\/ $2$ elements and in the
case\/ $\Phi=\C_l$, $l\ge 2$, assume additionally that any\/
$\theta\in R$ is contained in the ideal\/ $\theta^2R+2\theta R$.
\par

Then as a normal subgroup of the elementary Chevalley group\/ $E(\Phi,R)$,
the mixed commutator subgroup $[E(\Phi,R,I),E(\Phi,R,J)]$
is generated by the the elements of the form
\par\smallskip
$\bullet$ $\big[x_{\alpha}(\xi),z_{\alpha}(\zeta,\eta)\big]$,
\par\smallskip
$\bullet$ $\big[x_{\alpha}(\xi),x_{-\alpha}(\zeta)\big]$,
\par\smallskip
$\bullet$ $x_{\alpha}(\xi\zeta)$,
\par\smallskip\noindent
where $\alpha\in\Phi$, $\xi\in I$, $\zeta\in J$, $\eta\in R$.
\end{The}

Modulo level calculations, and the relative standard
commutator formula \cite{you92,RNZ2} this theorem easily implies
the following result.

\begin{The}\label{The:3}
Let $\Phi$ be a reduced irreducible root system of rank $\ge 2$.
In the cases\/ $\Phi=\C_2,\G_2$ assume that $R$ does not have residue
fields\/ ${\Bbb F}_{\!2}$ of\/ $2$ elements and in the
case\/ $\Phi=\C_l$, $l\ge 2$, assume additionally that any\/
$\theta\in R$ is contained in the ideal\/ $\theta^2R+2\theta R$.
\par
Further, let $I$ and $J$ be two ideals of a commutative ring $R$.
Then the mixed commutator subgroup
$\big[E(\Phi,R,I),E(\Phi,R,J)\big]$ is generated as a group by
the elements of the form
\par\smallskip
$\bullet$ $\big[x_{\alpha}(\xi),z_{\alpha}(\zeta,\eta)\big]$,
\par\smallskip
$\bullet$ $\big[x_{\alpha}(\xi),x_{-\alpha}(\zeta)\big]$,
\par\smallskip
$\bullet$ $z_{\alpha}(\xi\zeta,\eta)$,
\par\smallskip\noindent
where in all cases $\alpha\in\Phi$, $\xi\in I$, $\zeta\in J$,
$\eta\in R$.
\end{The}
In the proofs, the generating sets, described in these theorems,
will be denoted by $X$ and $Y$, respectively. Then the theorems
assert that
$$ \big[E(\Phi,R,I),E(\Phi,R,J)\big]={\langle X\rangle}^{E(\Phi,R)}=
{\langle Y\rangle}. $$
\noindent
Clearly, $X\subseteq Y$.
\par
For the case of the general linear group $\GL(n,A)$ over a
[not necessarily commutative] associative ring $A$, a similar set
of generators of mixed commutator subgroups as normal subgroups of
the absolute elementary group was constructed by the first and the
third author in \cite{RZ12}. Of course, in the non-commutative case
$[E(n,A,I),E(n,A,J)]\ge E(n,A,IJ+JI)$, so that one has to list
both $t_{ij}(\xi\zeta)$ and $t_{ij}(\zeta\xi)$, where
$1\le i\neq j\le n$, $\xi\in I$, $\zeta\in J$, as generators.
Later, this result was generalised to unitary groups in our
paper \cite{RNZ3}, see also \cite{RNZ4} for background, detailed
overview, and somewhat finer proofs.
\par
A slightly more delicate construction of generators of the mixed
commutator subgroups as subgroups was first carried through in
our joint paper with Alexei Stepanov \cite{HSVZmult}, again in
the simplest case of $\GL(n,R)$. As before in the non-commutative case
one has to engage both $z_{ij}(\xi\zeta,\eta)$ and
$z_{ij}(\zeta\xi,\eta)$ as generators.
\par
In the context of Chevalley groups, these results were announced
in our joint papers with Alexei Stepanov \cite{yoga2} and
\cite{portocesareo}. They are instrumental in Stepanov's proof
of the bounded width of relative commutators in terms of elementary
generators \cite{stepanov10}.
\par
The paper is organised as follows. In \S~\ref{sec:1} we recall
background facts that will be used in our proofs. In \S~2 we
implement a slightly stronger version of level calculations,
Theorem 4, further elaborating \cite{RNZ2}, Lemma 17. After
that, in \S~\ref{sec:3} we
are in a position to prove Theorem~\ref{The:2}, while in
\S~\ref{sec:4} we prove Theorem~\ref{The:3} and derive some of
its corollaries. Finally, in \S~\ref{sec:5} we discuss some further
aspects of relative commutator subgroups
$\big[E(\Phi,R,I),E(\Phi,R,J)\big]$, describe several cognate
results and applications, and formulate some unsolved problems.


\section{Some preliminary facts}\label{sec:1}

Here we collect some obvious or well known classical facts, which
will be used in our proofs. For background information on Chevalley
groups over rings, see \cite{NV91,vavplot}, where one can find many
further references.
\subsection{Commutator identities}
Let $G$ be a group. For any $x,y\in G$,  ${}^xy=xyx^{-1}$  denotes
the left $x$-conjugate of $y$. As usual, $[x,y]=xyx^{-1}y^{-1}$
denotes the [left normed] commutator of $x$ and $y$. We will make
constant use of the following obvious commutator identities, usually,
without any specific reference
\par\smallskip
(C1) $[x,yz]=[x,y]\cdot{}^y[x,z]$,
\par\smallskip
(C2) $[xy,z]={}^x[y,z]\cdot[x,z]$,
\par\smallskip
(C3) $[x,{}^yz]={}^y[{}^{y^{-1}}x,z]$,
\par\smallskip
(C4) $[{}^yx,z]={}^{y}[x,{}^{y^{-1}}z]$,
\par\smallskip
(C5) $[y,x]={[x,y]}^{-1}$.
\par\smallskip

Let $F,H\le G$ be two subgroups of $G$. By definition, the
mixed commutator subgroup $[F,H]$ is the subgroup generated by
all commutators $[x,y]$, where $x\in F$, $y\in H$. Clearly,
$[F,H]=[H,F]$. A subgroup $H\le G$ is normal in $G$ if $[G,H]\le H$.
A group $G$ is called perfect if $[G,G]=G$ and a subgroup $H\le G$
is called $G$-perfect if $[G,H]=H$.

\subsection{Steinberg relations} As in the introduction,
we denote by $x_{\a}(\xi)$, $\a\in\Phi$, $\xi\in R$ the
elementary generators of the [absolute] elementary subgroup
$E(\Phi,R)$. All results of the present paper are [directly or
indirectly] based on the Steinberg relations among the
elementary generators, which will be used without any
specific reference.
\par\smallskip
(R1) Additivity of $x_{\a}$,
$$ x_{\a}(\xi+\eta)=x_{\a}(\xi)x_{\a}(\eta). $$
\par
(R2) Chevalley commutator formula
$$ [x_{\a}(\xi),x_{\beta}(\eta)]=\prod_{i\a+j\beta\in \Phi}
x_{i\a+j\beta}(N_{\a\beta ij}\xi^i \eta^j), $$
\noindent
where $\a\not=-\beta$ and $N_{\a\beta ij}$ are the structure constants
which do not depend on $a$ and $b$. Notice, though, that for $\Phi=\G_2$
they may depend on the order of the roots in the product on the right
hand side.
\par
In particular, relation (R1) implies that $X_{\a}=\{x_{\a}(\xi),\xi\in R\}$
is a subgroup of $E(\Phi,R)$, isomorphic to $R^+$, called an
[elementary] root subgroup.

\subsection{Congruence subgroups}
Let $\rho_I:R\map R/I$ be the reduction modulo $I$. By functoriality,
it defines the group homomorphism $\rho_I:G(\Phi,R)\map G(\Phi,R/I)$.
The kernel of $\rho_I$ is denoted by $G(\Phi,R,I)$ and is called
the principal congruence subgroup of $G(\Phi,R)$ of level $I$.
In turn, the full pre-image of the centre of $G(\Phi,R/I)$ with
respect to the reduction homomorphism $\rho_{I}$ is called the
full congruence subgroup of level $I$, and
is denoted by $C(\Phi,R,I)$.
\par
A classical theorem by Andrei Suslin, Vyacheslav Kopeiko and
Giovanni Taddei, stated in the next subsection, asserts that $E(\Phi,R,I)$
is a normal subgroup of $G(\Phi,R)$. In particular, the quotient
$$ \K_1(\Phi,R,I)=G(\Phi,R,I)/E(\Phi,R,I), $$
\noindent
is a group and not just a pointed set.
\par
Theorem~\ref{The:2} is purely elementary, in the technical sense that all
calculations may take place inside $E(\Phi,R)$ and be performed
in terms of the Steinberg relations among elementary generators.
However, our proof of Theorem~\ref{The:3} is not elementary, in this sense,
and depends on the fact that, modulo the relative elementary
subgroup $E(\Phi,R,IJ)$, a mixed commutator subgroup
$\big[E(\Phi,R,I),E(\Phi,R,J)\big]$ lives inside $\K_1(\Phi,R,IJ)$.

\subsection{Standard commutator formula} In the present paper,
we use the following fundamental facts, known as the [absolute]
standard commutator formulas. The first proof of this result
by Giovanni Taddei \cite{taddei} and Leonid Vaserstein
\cite{vaser86} was based on localisation. We refer the reader
to \cite{RN1} for a slightly easier proof, based on similar
ideas, and to \cite{RN} for many further references, pertaining
to the classical cases, and other proofs.

\begin{Lem}\label{Lem:1} Let\/ $\rk(\Phi)\ge 2$. Then for any
ideal\/ $I$ of the ring\/ $R$ one has the following inclusions
$$ \big[E(\Phi,R,I),G(\Phi,R)\big],
\big[G(\Phi,R,I),E(\Phi,R)\big]\le E(\Phi,R,I). $$
\end{Lem}
It is classically known, that the only cases, where $E(\Phi,R)$
is not perfect --- or, for that matter, where $E(\Phi,R,I)$ is
not $E(\Phi,R)$-perfect --- are groups of types $\C_2$ and $\G_2$
over a ring having residue fields ${\Bbb F}_{\!2}$ of two
elements. Thus,
\begin{Lem}\label{Lem:2} Let\/ $\rk(\Phi)\ge 2$. In the case
$\Phi=\C_2,\G_2$ assume additionally that $R$ does not have
residue fields ${\Bbb F}_{\!2}$ of two elements. Then for any
ideal\/ $I$ of the ring\/ $R$ one has the following inclusions
$$ \big[E(\Phi,R,I),G(\Phi,R)\big]=
\big[G(\Phi,R,I),E(\Phi,R)\big]=E(\Phi,R,I). $$
\end{Lem}
Actually, in the final section we recall a {\it birelative\/}
standard commutator formula, which is a simultaneous generalisation of
both standard formulas above. This formula, and attempts to generalise
it, served as the main motivation behind the results of the present
paper. However, it is not directly used in their proofs.

\subsection{Elementary subgroups of level $I$} Let $I$ be an ideal
of $R$. In the sequel we use also the following non relativised
version of elementary subgroups of level $I$:
$$ E(\Phi,I)={\big\langle x_{\alpha}(\xi)\mid \alpha\in\Phi,\
\xi\in R\big\rangle} $$
\noindent
By definition, $E(\Phi,R,I)={E(\Phi,I)}^{E(\Phi,R)}$ is the normal
closure of $E(\Phi,I)$ in the absolute elementary group $E(\Phi,R)$,
and thus, normal in $G(\Phi,R)$.
\par

\begin{Lem}\label{Lem:3}
Let\/ $\rk(\Phi)\ge 2$ and further let\/ $I$ and\/ $J$ be two
ideals of\/ $R$. Assume that either $\Phi\neq\C_l$, or $2\in R^*$.
Then one has $E(\Phi,R,IJ)\le E(\Phi,I+J)$.
\end{Lem}
To state an analogue of this inclusion in the exceptional case,
we should distinguish the ideal $I^2$, generated by the products
$ab$, where $a,b\in I$, from the ideal $I\supb{2}$, generated
by $a^2$, where $a\in I$. Clearly, when $2\in R^*$ these ideals
coincide, but this case is trivial anyway. In general one can only
guarantee a weaker inclusion $E(\Phi,R,I\supb{2}J+2IJ+IJ\supb{2})\le
E(\Phi,I+J)$.
In fact, both the generic and the exceptional cases are special
instances of the following more precise statement,
$$ E(\Phi,R,IJ,I\supb{2}J+2IJ+IJ\supb{2})\le E(\Phi,I+J), $$
\noindent
see \cite{RNZ2}, Lemma 16. Here, the left hand side is the relative
elementary subgroups corresponding to the admissible pair
$(IJ,I\supb{2}J+2IJ+IJ\supb{2})$, see \cite{abe69,abe89}, or
\cite{HPV} for further references.

\subsection{Unitriangular subgroups} Further, we fix an order on
$\Phi$ and denote by $\Pi=\{\a_1,\ldots,\a_l\}$, $\Phi^+$ and
$\Phi^-$ the corresponding sets of fundamental, positive and
negative roots, respectively.
\par
The unipotent radicals of the opposite Borel subgroups, standard
with respect to this order, are defined as follows:
$$ \begin{aligned}
U&=U(\Phi,R)=
\big\langle x_\a(\xi),\ \a\in\Phi^+,\ \xi\in R\big\rangle, \\
\noalign{\vskip 4pt}
U^-&=U^-(\Phi,R)=
\big\langle x_\a(\xi),\ \a\in\Phi^-,\ \xi\in R\big\rangle. \\
\end{aligned} $$
\noindent
By definition, $E(\Phi,R)=\langle U(\Phi,R),U^-(\Phi,R)\rangle$.
\par
The Chevalley commutator formula immediately implies the following
result, see for instance, \cite{carter} or \cite{steinberg}.

\begin{Lem}\label{Lem:4}
The group $U(\Phi,R)$ admits factorisation $U(\Phi,R)=\prod X_{\a}$,
over all positive roots $\a\in\Phi^+$, taken in any prescribed
order. For a given order, an expression of an element $u\in U(\Phi,R)$
in the form $u=\prod x_{\a}(u_\a)$, where $u_{\a}\in R$, is
unique.
\end{Lem}

In turn, uniqueness of factors in Lemma~\ref{Lem:4} immediately implies the
following result, which will be used in the proof of Theorem~\ref{The:2}.

\begin{Lem}\label{Lem:5}
For any ideal $I$ of $R$ one has
$$ G(\Phi,R,I)\cap U(\Phi,R)\le E(\Phi,I). $$
\end{Lem}

It is easy to see this also uses representations of $G(\Phi,R)$.
Namely, one considers a non-trivial rational representation
$\pi:G(\Phi,R)\map\GL(n,R)$ such that $T(\Phi,R)$ is represented
by diagonal matrices, $U(\Phi,R)$ is represented by upper
unitriangular matrices, and $U^-(\Phi,R)$ is represented by lower
unitriangular matrices. If $\pi$ is faithful, one has
$G(\Phi,R,I)=\pi^{-1}(\pi(G(\Phi,R))\cap\GL(n,R,I))$, see, for
instance, \cite{HPV}, Lemma~6. In particular, a product
$\prod x_{\a}(\xi_{\a})$, $\a\in\Phi^+$, belongs to $G(\Phi,R,I)$
if and only if $\xi_{\a}\in I$ for all $\a\in\Phi^+$.

\par
\subsection{Parabolic subgroups}
The main role in our simplified proof of Theorem~\ref{The:2} is played
by the Levi decomposition for [elementary] parabolic subgroups.
Classically, it asserts that any parabolic subgroup $P$ of
$G(\Phi,R)$ can be expressed as the semi-direct product
$P=L_P\rightthreetimes U_P$ of its unipotent radical
$U\unlhd P$ and a Levi subgroup $L_P\le P$. However, we do not
wish to recall the requisite background, to state the general
case. Luckily, this is not at all nesessary.
\par\smallskip
$\bullet$ Since we calculate inside $E(\phi,R)$, we can limit
ourselves to the {\it elementary\/} parabolic subgroups,
spanned by some root subgroups $X_{\a}$.
\par\smallskip
$\bullet$ Since we can choose the order on $\Phi$ arbitrarily,
we can further limit ourselves to {\it standard\/} parabolic
subgroups.
\par\smallskip
$\bullet$ Since the proof of Theorem~\ref{The:2} easily reduces to groups
of rank 2, we could only consider rank 1 parabolic subgroups,
which {\it in this case\/} coincide with maximal parabolic
subgroups.
\par\smallskip
It is notationally easier to describe rank 1 parabolics, which
in discrepancy with the usage of the theory of algebraic groups,
will be called {\it minimal\/} parabolics. Namely, we fix an
$r$, $1\le r\le l$, and consider the subgroup
$$ P_r=\langle U,X_{\a_r}\rangle\le E(\Phi,R), $$
\noindent
which we call the $r$-th {\it elementary\/} minimal standard
parabolic subgroup of $E(\Phi,R)$. Further, set
$$ U_r=\prod X_{\a},\qquad \a\in\Phi^+,\quad \a\neq\a_r; $$
\noindent
this is the unipotent radical of $P_r$. Finally, let
$L_r=\langle X_{\a_r},X_{-\a_r}\rangle$ be the [standard] Levi subgroup
of $P_r$.
\par
The following well result is an immediate corollary of the general
Levi decomposition. However, this special instance of Levi
decomposition can be easily derived from the Chevalley commutator
formula, without any reference to the theory of algebraic groups.

\begin{Lem}\label{Lem:7}
For any $1\le r\le\rk(\Phi)$ the group $P_r$ is the semi-direct
product $P_r=L_r\rightthreetimes U_r$ of its unipotent radical
$U_r\unlhd P_r$ and its Levi subgroup $L_r\le P_r$.
\end{Lem}

The most important part is the [obvious] claim that $U_r$
is normal in $P_r$.

\par


\section{Levels of mixed commutators}\label{sec:2}

The following easy level calculation was performed on several
occasions in various contexts,
see for instance, Theorem 1 of Hong You \cite{you92} for a
slightly weaker statement\footnote{With an extra-assumption
in the case $\Phi=\G_2$, and a missing assumption in the case of
$\Phi=\C_l$, $l\ge 3$.}. In this form, it is essentially a
combination of [a slightly stronger version of] Lemma~17 and
of Lemma~19 in \cite{RNZ2}. The
left-most inclusion serves to motivate the statements of
Theorems~\ref{The:2} and \ref{The:3}, and to verify
that the $X$ and $Y$ are actually contained in
$\big[E(\Phi,R,I),E(\Phi,R,J)\big]$, whereas the right-most
inclusion will be used to derive Theorem~\ref{The:3} from
Theorem~\ref{The:2}.

\begin{The}\label{The:4}
Let\/ $\Phi$ be a reduced irreducible root system of rank\/ $\ge 2$.
In the cases\/ $\Phi=\C_2,\G_2$ assume that $R$ does not have residue
fields\/ ${\Bbb F}_{\!2}$ of\/ $2$ elements and in the
case\/ $\Phi=\C_l$, $l\ge 2$, assume additionally that any\/
$\theta\in R$ is contained in the ideal\/ $\theta^2R+2\theta R$.
\par
Then for any two ideals\/ $I$ and\/ $J$
of the ring\/ $R$ one has the following inclusion
\begin{align*}
E(\Phi,R,IJ)\le\big[E(\Phi,I),E(\Phi,J)\big]&\le
\big[E(\Phi,R,I),E(\Phi,R,J)\big]\\
&\le\big[G(\Phi,R,I),G(\Phi,R,J)\big]\le G(\Phi,R,IJ).
\end{align*}
\end{The}

In the cases $\Phi=\C_2$ and $\Phi=\G_2$, the left-most inclusion
may fail without these additional assumptions. For rings having
residue field ${\Bbb F}_{\!2}$, this was observed in connection with
Lemma~\ref{Lem:2}. Without additional assumption, in the case
$\Phi=\C_2$ one can only claim that
$$ E(\Phi,R,IJ,I\supb{2}J+2IJ+I\supb{2}2)\le\big
[E(\Phi,R,I),E(\Phi,R,J)\big], $$
\noindent
In other words, one has to distinguish the levels of short and long
roots, and modify the generating sets in Theorems~\ref{The:2} and~\ref{The:3} accordingly.
Since the case of $\Sp(4,R)$ requires a separate analysis anyway,
in the present paper we do not pursue this any further, not to
overcharge the statements and proofs with such technical details,
peculiar for this particular case.

\begin{proof}
All inclusions, apart from the left-most one, are either obvious,
or established in \cite{RNZ2}, Lemmas~17 and~19. However, in
the proof of \cite{RNZ2}, Lemma~17, we only demonstrated a weaker
inclusion
$$ E(\Phi,R,IJ)\le\big[E(\Phi,R,I),E(\Phi,R,J)\big]. $$
\noindent
Thus, it remains to verify that $E(\Phi,R,IJ)$ is contained already
in $\big[E(\Phi,I),E(\Phi,J)\big]$. This can be done along the same
lines, as in \cite{RNZ2}, but requires some extra care. For the sake of brevity, in the sequel we denote $\big[E(\Phi,I),E(\Phi,J)\big]$ by $H$.
\par
Namely, it suffices to show that the generators
of the group $E(\Phi,R,IJ)$ belong to $ H$.
Indeed, let
$$ z_{\a}(\xi\zeta,\eta)={}^{x_{-\a}(\eta)}x_{\a}(\xi\zeta),\qquad
\xi\in I,\quad\zeta\in J,\quad\eta\in R, $$
\noindent
be one of these generators. However, in \cite{RNZ2} the right hand
side of the inclusion was already normal  in $E(\Phi,R)$, so that
we merely had to verify that
$x_{\a}(\xi\zeta)\in H$. Here, at this
stage we do not yet know that the right hand side is normal --- this
is exactly what we are attempting to prove! Thus, now at each step we
have to ascertain that all occurring extra factors are actually
contained in $ H$.
\par\smallskip
$\bullet$ First, assume that $\a$ can be embedded in a root
system of type $\A_2$. Then there exist roots $\beta,\gamma\in\Phi$,
of the same length as $\a$ such that $\a=\beta+\gamma$, and
$N_{\beta\gamma 11}=1$. By the above, we can express
$x_{\a}(\xi\zeta)$ as
$$ x_{\a}(\xi\zeta)=[x_{\beta}(\xi),x_{\gamma}(\zeta)]\in
 H. $$
\noindent
Thus,
\begin{multline*}
z_{\a}(\xi\zeta,\eta)=
{}^{x_{-\a}(\eta)}[x_{\beta}(\xi),x_{\gamma}(\zeta)]=
[{}^{x_{-\a}(\eta)}x_{\beta}(\xi),{}^{x_{-\a}(\eta)}x_{\gamma}(\zeta)]=\\
[x_{\beta}(\xi)x_{-\gamma}(\pm\xi\eta),
x_{\gamma}(\zeta)x_{-\beta}(\pm\zeta)\eta]\in
 H,
\end{multline*}
\noindent
as claimed.
\par
This proves the lemma for simply laced Chevalley groups, and for the
Chevalley group of type $\F_4$. It also proves necessary inclusions for
{\it short\/} roots in Chevalley groups of type $\C_l$, $l\ge 3$,
for {\it long\/} roots in Chevalley groups of type $\B_l$, $l\ge 3$,
and for {\it long\/} roots in the Chevalley group of type $\G_2$.
\par
For the remaining cases, the idea of the proof is similar, but
requires more care, due to the more complicated form of the
Chevalley commutator formula.
\par\smallskip
$\bullet$ Next, assume that $\a$ can be embedded in a root
system of type $\C_2$ as a {\it long\/} root. We wish to prove that
$z_{\a}(\xi\zeta,\eta)\in H$. As the first
approximation, we prove that
$$ z_{\a}(\xi^2\zeta,\eta),z_{\a}(\xi\zeta^2,\eta)\in
 H. $$
\noindent
There exist roots $\beta,\gamma\in\Phi$, such that $\a=\beta+2\gamma$
and $N_{\beta\gamma 11}=1$. Clearly, in this case $\beta$ is long and
$\gamma$ is short. Take an arbitrary $\theta\in R$. Then
$$ [x_{\beta}(\theta\xi),x_{\gamma}(\zeta)]=
x_{\beta+\gamma}(\theta\xi\zeta)x_{\a}(\pm \theta\xi\zeta^2)\in
 H, $$
\noindent
whereas
$$ [x_{\beta}(\xi),x_{\gamma}(\theta\zeta)]=
x_{\beta+\gamma}(\theta\xi\zeta)x_{\a}(\pm \theta^2\xi\zeta^2)\in
 H. $$
\noindent
Comparing these two inclusions we can conclude that
\begin{multline*}
z_{\a}((\theta^2-\theta)\xi\zeta^2,\eta)=
{}^{x_{-\a}(\eta)}[x_{\beta}(\xi),x_{\gamma}(\theta\zeta)]\cdot
{}^{x_{-\a}(\eta)}[x_{\gamma}(\zeta),x_{\beta}(\theta\xi)]=\\
[{}^{x_{-\a}(\eta)}x_{\beta}(\xi),{}^{x_{-\a}(\eta)}x_{\gamma}(\theta\zeta)]\cdot
[{}^{x_{-\a}(\eta)}x_{\gamma}(\zeta),{}^{x_{-\a}(\eta)}x_{\beta}(\theta\xi)]=\\
[x_{\beta}(\xi),x_{\gamma}(\theta\zeta)x_{-\beta-\gamma}(\pm\theta\zeta\eta)
x_{-\beta}(\pm\eta\theta^2\zeta^2)]\cdot\\
[x_{\gamma}(\zeta)x_{-\beta-\gamma}(\pm\zeta\eta)
x_{-\beta}(\pm\eta\zeta^2),x_{\beta}(\theta\xi)]\in
 H.
\end{multline*}
\noindent
By assumption $R$ does not have residue field of $2$ elements,
and thus the ideal generated by $\theta^2-\theta$, where $\theta\in R$,
is not contained in any maximal ideal, and thus coincides with $R$.
Now, since $z_{\a}(\lambda+\mu,\eta)=z_{\a}(\lambda,\eta)z_{\a}(\mu,\eta)$,
while $I$ and $J$ are ideals of $R$, it
follows that $z_{\a}(\xi\zeta^2,\eta)\in H$.
Interchanging $\xi$ and $\zeta$ we see that
$z_{\a}(\xi^2\zeta,\eta)\in H$.
\par
Now, we can pull the left bootstrap, before returning again to
the right one.
\par\smallskip
$\bullet$ Namely, assume that $\a$ can be embedded in a root system of
type $\C_2$ as a {\it short\/} root. Choose $\beta$ and $\gamma$
subject to the same condition, as in the preceding item. In other words,
$\a=\beta+\gamma$, $\beta$ is long, $\gamma$ is short, and
$N_{\beta\gamma 11}=1$. Then, clearly
$$ z_{\a}(\xi\zeta,\eta)={}^{x_{-\a}(\eta)}
[x_{\beta}(\xi),x_{\gamma}(\zeta)]
\cdot {}^{x_{-\a}(\eta)}x_{\alpha+\gamma}(\mp\xi\zeta^2). $$
\noindent
Pulling conjugation by $x_{-\a}(\eta)$ inside the first commutator,
we get
$$ [{}^{x_{-\a}(\eta)}x_{\beta}(\xi),{}^{x_{-\a}(\eta)}x_{\gamma}(\zeta)]=
[x_{\beta}(\xi)x_{-\gamma}(\pm\xi\eta)x_{-\alpha-\gamma}(\pm\xi\eta^2),
x_{\gamma}(\zeta)x_{-\beta}(\pm\zeta\eta)]
\in H. $$
\noindent
On the other hand, from the previous item we already know that
$x_{\alpha+\gamma}(\mp\xi\zeta^2)$ can be expressed as a product
of {\it some\/} commutators of the form
$[x_{\beta}(\rho),x_{\gamma}(\sigma)]$, where $\rho\in I$,
$\sigma\in J$, {\it provided\/} that $R$ does not have residue
field of $2$ elements. Exactly the same calculation as in the
above display line shows that the conjugates of these commutators
by $x_{-\a}(\eta)$ remain inside $ H$.
Observe, that from the first item, we already know that for
$\Phi=\B_l$, $l\ge 3$, even the stronger inclusion
$z_{\a+\gamma}(\xi\zeta,\eta)\in H$
holds without any such assumption.
\par
Thus, in both cases we can conclude that
$z_{\a}(\xi\zeta,\eta)\in H$
for a {\it short\/} root $\alpha$. Again, already from the first item
we know that for $\Phi=\C_l$, $l\ge 3$, this inclusion holds without
any assumptions on $R$.
\par\smallskip
On the other hand, a {\it long\/} root $\alpha$ of $\Phi=\C_l$,
$l\ge 3$, cannot be embedded in an irreducible rank 2 subsystem
other than $\C_2$. This leaves us with the analysis of exactly two
rank 2 cases: $\Phi=\C_2$ and $\a$ is long and $\Phi=\G_2$ and $\a$
is short. This is where one needs the additional assumptions on $R$.
We pull the right bootstrap once more, now more tightly.
\par\smallskip
$\bullet$ Let $\Phi=\C_2$ and $\a$ is long. Then $\a$ can be expressed
as $\a=\beta+\gamma$ for two {\it short\/} roots $\beta$ and $\gamma$.
Interchanging $\beta$ and $\gamma$ we can assume that
$N_{\beta\gamma 11}=2$. Then one has
$$ x_{\a}(2\xi\zeta)=[x_{\beta}(\xi),x_{\gamma}(\zeta)]\in
 H. $$
\noindent
Thus, we can conclude that
\begin{multline*}
z_{\a}(2\xi\zeta,\eta)=
[{}^{x_{-\a}(\eta)}x_{\beta}(\eta),{}^{x_{-\a}(\eta)}x_{\gamma}(\zeta)]=\\
[x_{\beta}(\xi)x_{-\gamma}(\pm\xi\eta)x_{\beta-\gamma}(\pm\xi^2\eta),
x_{\gamma}(\zeta)x_{-\beta}(\pm\zeta\eta)
x_{\gamma-\beta}(\pm\zeta^2\eta)]\in
 H.
\end{multline*}
\noindent
One the other hand, from the second item we already know that
$z_{\a}(\xi^2\zeta)\in H$.
Since by assumption the ideal generated by $2\xi$ and $\xi^2$
contains $\xi$, no exactly the same arhument as in the second
item shows that $z_{\a}(\xi,\zeta,\eta)\in H$.
\par\smallskip
$\bullet$ Finally, let $\Phi=\G_2$ and $\a$ is short. We can argue
essentially in the same way as for the case of $\Phi=\C_2$.
In a sense, it is slightly more complicated due to the fancier
form of the Chevalley commutator formula. On the other hand,
overall it is even easier, since we already have the necessary
inclusions for {\it long\/} roots.
\par
Again, as the first approximation, we prove that
$z_{\a}(\xi^2\zeta,\eta),z_{\a}(\xi\zeta^2,\eta)\in H$.
With this end, express $\a$
as $\a=\beta+2\gamma$, where $\beta$ is short, $\gamma$ is long,
and $N_{\beta\gamma 11}=1$. Take an arbitrary $\theta\in R$. Then
$$ [x_{\beta}(\theta\xi),x_{\gamma}(\zeta)]=
x_{\beta+\gamma}(\theta\xi\zeta)x_{\a}(\pm \theta^2\xi^2\zeta)
x_{3\beta+\gamma}(\pm\theta^3\xi^3\zeta)
x_{3\beta+2\gamma}(\pm\theta^3\xi^3\zeta^2)\in
 H, $$
\noindent
whereas
$$ [x_{\beta}(\xi),x_{\gamma}(\theta\zeta)]=
x_{\beta+\gamma}(\theta\xi\zeta)x_{\a}(\pm \theta\xi^2\zeta)
x_{3\beta+\gamma}(\pm\theta\xi^3\zeta)
x_{3\beta+2\gamma}(\pm\theta^2\xi^3\zeta^2)\in
 H. $$
\noindent
Since the roots $3\beta+\gamma$ and $3\beta+2\gamma$ are long, from
the first item we already know that the corresponding root elements
belong to $ H$, and, in fact, are already
expressed as commutators $[x_{\phi}(\rho),x_{\psi}(\sigma)]$,
where $\rho\in I$ and $\sigma\in J$, of two {\it long\/} root
elements in $E(\A_2,I)$ and $E(\A_2,J)$. Since $\a$ is {\it short\/},
by the usual argument, conjugation by $x_{-\a}(\eta)$ does not
move such a commutator out of $ H$.
\par
This means, that when comparing the above inclusions, as we did in
the second item above, we can ignore all overhanging long root
factors, and argue modulo the subgroup $[E(\A_2,I),E(\A_2,J)]$.
This leaves us with the inclusion
$$ x_{\a}(\pm(\theta^2-\theta)\xi\zeta)\in
[x_{\beta}(\xi),x_{\gamma}(\theta\zeta)]
[x_{\gamma}(\zeta),x_{\beta}(\theta\xi)]\big[E(\A_2,I),E(\A_2,J)\big]. $$
\noindent
Conjugating this inclusion by $x_{-\a}(\eta)$, carrying the
congujation inside the commutator, and observing that $\a$ is
distinct from $\b$ and $\g$, we can in the usual way conclude that
$z_{\alpha}(\pm(\theta^2-\theta)\xi^2\zeta,\eta)\in H$,
for all $\theta\in R$. Again, since $R$ does not have residue field
of two elements, we can repeat the same argument as in the second
item to derive that $z_{\alpha}(\xi^2\zeta,\eta)\in H$.
Interchanging $\xi$ and $\zeta$, we see that the same holds also
for $z_{\a}(\xi\zeta^2,\eta)$.
\par
To conclude the proof it only remains to look at another short root.
Namely, set $\a=\beta+\gamma$, for the same roots $\beta$ and
$\gamma$, as above. Looking at the commutator
$$ [x_{\beta}(\xi),x_{\gamma}(\zeta)]=
x_{\a}(\xi\zeta)x_{2\beta+\gamma}(\pm \xi^2\zeta)
x_{3\beta+\gamma}(\pm \xi^3\zeta)x_{3\beta+2\gamma}(\pm \xi^3\zeta^2)\in H, $$
\noindent
and carrying all factors, apart from the first one, from the right
hand side to the left hand side, we get an expression of $x_{\a}(\xi\zeta)$
as a product of commutators of the form $[x_{\phi}(\rho),x_{\psi}(\sigma)]$,
where either $\{\phi,\psi\}=\{\beta,\gamma\}$, or else both
$\phi$ and $\psi$ are long. As we have already observed, conjugation
by $x_{-\a}(\eta)$ does not move such a commutator out of $H$.
This means that $z_{\a}(\xi\zeta,\eta)\in H$, as claimed.
\end{proof}


\section{Proof of Theorem~\ref{The:2}}\label{sec:3}

As a group, the relative commutator subgroup $\big[E(\Phi,R,I),E(\Phi,R,J)\big]$
is generated by the commutators $[x,y]$, where $x\in E(\Phi,R,I)$
and $y\in E(\Phi,R,J)$. By Theorem~\ref{The:1} we can express $x$ as a
product of generators $z_{\a}(\xi,\eta)$, where $\a\in\Phi$,
$\xi\in I$, $\eta\in R$. Similarly, we can express $y$ as a
product of generators $z_{\b}(\zeta,\theta)$, where $\b\in\Phi$,
$\zeta\in I$, $\theta\in R$.
\par
Now, iterated application of commutator identities C1 and C2
implies that as a group
$\big[E(\Phi,R,I),E(\Phi,R,J)\big]$ is generated
by the conjugates ${}^g[z_{\a}(\xi,\eta),z_{\b}(\zeta,\theta)]$,
where $g\in E(\Phi,R)$, whereas $\a,\b,\xi,\zeta,\eta,\theta$
have the same sense, as above.
\par
Next, observe that
$$
{\vphantom{\big[]}}^g\big[z_{\a}(\xi,\eta),z_{\b}(\zeta,\theta)\big]=
{\vphantom{\big[]}}^g\big[{}^{x_{-\a}}(\eta)x_{\a}(\xi),z_{\b}(\zeta,\theta)]=
{\vphantom{\big[]}}^{gx_{-\a}(\eta)}\big[x_{\a}(\xi),
{}^{x_{-\a}(-\eta)}z_{\b}(\zeta,\theta)]. $$
\noindent
Since $E(\Phi,R,J)$ is normal in $E(\Phi,R)$, it follows that
${}^{x_{-\a}(-\eta)}z_{\b}(\zeta,\theta)$ can be expressed as a
product of generators $z_{\g}(\lambda,\rho)$, where $\g\in\Phi$,
$\lambda\in J$, $\rho\in R$.
\par
Again, iterated application of commutator identities C1 and C2
implies that $\big[E(\Phi,R,I),E(\Phi,R,J)\big]$ is generated
by a smaller set of conjugates ${}^g[x_{\a}(\xi),z_{\b}(\zeta,\theta)]$,
where $g\in E(\Phi,R)$, whereas $\a,\b,\xi,\zeta,\theta$
have the same sense, as above.
\par
Thus, it suffices to prove that the commutators
$[x_{\a}(\xi),z_{\b}(\zeta,\theta)]$ are expressed as products
of {\it conjugates\/} of elements from $X$. Now, the proof depends
on the mutual position of the roots $\a$ and $\b$.
\par\smallskip
$\bullet$ When $\a=\b$, there is nothing to prove, since these
commutators themselves belong to $X$. These are generators of the
first type.
\par\smallskip
$\bullet$ When $\a=-\b$, one has
$$ \big[x_{\a}(\xi),z_{-\a}(\zeta,\theta)\big]=
\big[x_{\a}(\xi),{}^{x_{\a}(\theta)}x_{-\a}(\zeta)\big]=
{}^{x_{\a}(\theta)}\big[x_{\a}(\xi),x_{-\a}(\zeta)\big]. $$
\noindent
In other words, these commutators are themselves conjugates of
individual elements of $X$, namely, of the generators of the
second type.
\par\smallskip
Thus, in the sequel we may assume that $\a\neq\pm\b$.
\par\smallskip
$\bullet$ If $\a$ and $\b$ are strictly orthogonal, in other
words, $\a\pm\b\notin\Phi$, then
$$ [x_{\a}(\xi),x_{\b}(\zeta)\big]=[x_{\a}(\xi),x_{-\b}(\theta)\big]=e, $$
\noindent
and thus also $[x_{\a}(\xi),z_{\b}(\zeta,\theta)]=e$.
\par\smallskip
$\bullet$ Thus, we may assume that $\a$ and $\b$ generate an irreducible
root system $\Delta$ of rank 2, viz.\ $\A_2$, $\B_2$ or $\G_2$. Thus,
it suffices to prove the desired fact for Chevalley groups of rank 2.
\par
Clearly, we may choose an order on $\Delta$ in such a way that either
$\b$ or $-\b$ is a fundamental root, whereas $\a\in\Delta^+$. Consider
the [maximal] parabolic subgroup $P_{\b}\le G(\Delta,R)$, corresponding
to the parabolic set $\Delta^+\cup\{\pm\b\}$. Then $z_{\b}(\zeta,\theta)$
belongs to the elementary subgroup $\langle X_{\b},X_{-\b}\rangle$ of
its Levi subgroup $L_{\b}=G_{\b}$. On the other hand, $x_{\a}(\xi)$
belongs to its unipotent radical
$$ U_{\b}=\langle U_{\g},\ \g\in\Delta^+, \g\neq\pm\b\rangle. $$
\par
Recalling that by Lemma~\ref{Lem:7} the unipotent radical $U_{\b}$ is a normal
subgroup of $P_{\b}$, we see that
$$ [x_{\a}(\xi),z_{\b}(\zeta,\theta)]\in [U_{\b},L_{\b}]\le
U_{\b}\le U(\Delta,R). $$
\par
On the other hand, Lemma~\ref{Lem:2} implies that
$$ [x_{\a}(\xi),z_{\b}(\zeta,\theta)]\in G(\Delta,R,IJ). $$
\par
Combining these inclusions, we see that
$$ [x_{\a}(\xi),z_{\b}(\zeta,\theta)]\in
G(\Delta,R,IJ)\cap U(\Delta,R). $$
\noindent
By Lemma~\ref{Lem:5} this intersection is contained in $E(\Delta,IJ)$. In
particular, commutators of this type are expressed as products of
elements $x_{\g}(\lambda\mu)$, where $\g\in\Delta^+$, $\lambda\in I$,
$\mu\in J$. But this is precisely the third type of generators.
\par
Conversely, Lemma~\ref{The:4} implies that these generators belong to
the mixed commutator subgroup $\big[E(\Phi,R,I),E(\Phi,R,J)\big]$.
This finishes the proof of Theorem~\ref{The:2}.


\section{Proof of Theorem~\ref{The:3} and some corollaries}\label{sec:4}

It turns out, that Theorem~\ref{The:3} easily follows from Theorem~\ref{The:2},
modulo other known results on relative commutator subgroups.

\begin{proof}[{Proof of Theorem~$\ref{The:3}$}]
Clearly, $X\subseteq Y$, so that by Theorem~\ref{The:2} the set $Y$ generates
$\big[E(\Phi,R,I),E(\Phi,R,J)\big]$ as a {\it normal\/} subgroup
of $E(\Phi,R)$. Therefore, it suffices to show that conjugates
of the elements of $Y$ above generators are themselves products of
elements of $Y$. Let $g\in Y$ and let $h\in E(\Phi,R)$. By
Theorem~\ref{The:4}
one has $g\in G(\Phi,R,IJ)$. Now applying Lemma~\ref{Lem:2} we see that
$$ [h,g]\in[G(\Phi,R,IJ),E(\Phi,R)]=E(\Phi,R,IJ). $$
\noindent
Thus, $hgh^{-1}=gz$, for some $z\in E(\Phi,R,IJ)$. Invoking Theorem~\ref{The:1}
we see that $hgh^{-1}$ is the product of $g$ and some
$z_{\a}(\xi\zeta,\eta)$, where $\a\in\Phi$, $\xi\in I$, $\zeta\in J$,
$\eta\in R$. All of these factors belong to $Y$. This completes the proof.
\end{proof}

A closer look at the generators in Theorem~\ref{The:3} shows that all of them
in fact belong already to $\big[E(\Phi,I),E(\Phi,R,J)\big]$. By
symmetry, we may switch the role of factors. In particular, this
means that Theorem~\ref{The:3} implies the following curious fact.

\begin{Cor}\label{Cor:1}
Under assumptions of Theorem~\ref{The:3} one has
$$ \big[E(\Phi,I),E(\Phi,R,J)\big]=\big[E(\Phi,R,I),E(\Phi,J)\big]=
\big[E(\Phi,R,I),E(\Phi,R,J)\big]. $$
\end{Cor}

Let us mention another amusing corollary of our results, which was
not noted before. Unlike the slightly larger commutators in the
previous corollary, it is not clear, why $ H$
should be equal to $\big[E(\Phi,R,I),E(\Phi,R,J)\big]$, but it is
at least normal in the absolute elementary group.

\begin{Cor}\label{Cor:2}
Under assumptions of Theorem~\ref{The:3} the mixed commutator subgroup\/
$\big[E(\Phi,I),E(\Phi,J)\big]$ is normal in\/ $E(\Phi,R)$.
\end{Cor}
\begin{proof}
By Theorem~\ref{The:4}, we have
$$ E(\Phi,R,IJ)\le \big[E(\Phi,I),E(\Phi,J)\big]\le G(\Phi,R,IJ). $$
\noindent
By taking the commutator of each component of the above inclusion with
$E(\Phi, R)$, the standard commutator formula implies that
$$ \big[E(\Phi,R,IJ),E(\Phi,R)\big]=
\big[E(\Phi,R,IJ),E(\Phi, R)\big]=E(\Phi,R,IJ). $$
\noindent
It follows immediately that
$$ \Big[\big[E(\Phi,I),E(\Phi,J)\big],E(\Phi, R)\Big]=
E(\Phi,R,IJ)\le \big[E(\Phi,I),E(\Phi,J)\big]. $$
\noindent
Thus, $\big[E(\Phi,I),E(\Phi,J)\big]$ is normal in\/ $E(\Phi,R)$.
\end{proof}

\section{Final remarks}\label{sec:5}

The relative commutator subgroups $\big[E(\Phi,R,I),E(\Phi,R,J)\big]$
considered in the present paper seem to be a very interesting class
of subgroups. They occur unreasonably often in many seemingly
unrelated situations.
\par\smallskip
$\bullet$ For the general linear group these commutator subgroups,
as also the commutator subgroups of the principal congruence
subgroups and full congruence subgroups
$$ \big[\GL(n,R,I),\GL(n,R,J)\big],\quad
\big[\GL(n,R,I),C(n,R,J)\big],\quad
\big[C(n,R,I),C(n,R,J)\big], $$
\noindent
were first considered by Hyman Bass \cite{Bass2}, and then systematically
studied by Alec Mason and Wilson Stothers \cite{MAS3,Mason74,MAS1,MAS2}.
In particular, they established the formula
$$ \big[\GL(n,R,I),\GL(n,R,J)\big]=[E(n,R,I),E(n,R,J)], $$
\noindent
provided that $n\ge\sr(R)+1,3$.
\par\smallskip
$\bullet$ The following result, the relative standard commutator formula
is the main result of You Hong \cite{you92} and the main application of
the relative commutator calculus developed by us in \cite{RNZ2}.

\begin{The}[{You Hong, Hazrat---Vavilov---Zhang}]\label{The:5}
Let\/ $\Phi$ be a reduced irreducible root system,\/ $rk(\Phi)\ge 2$.
Further, let\/ $R$ be a commutative ring, and\/ $I,J\trianglelefteq R$
be two ideals of\/ $R$. Then
$$ [E(\Phi,R,I),C(\Phi,R,J)]=[E(\Phi,R,I),E(\Phi,R,J)]. $$
\end{The}

For the general linear group, we have three different proofs of this
formula, based on decomposition of unipotents \cite{NVAS},
localisation \cite{RZ11}, and on level calculations \cite{NVAS2}.
Later, we generalised the last two of these proofs to Bak's unitary
groups \cite{RNZ1} and to Chevalley groups \cite{RNZ2}.

\par\smallskip
$\bullet$ Let us state an amazing application of Theorem~\ref{The:5} and our
main Theorem~\ref{The:3}, which is due to Alexei Stepanov \cite{stepanov10},
see also \cite{yoga2,portocesareo}.

\begin{The}[{Stepanov}]\label{The:6}
Let\/ $R$ be a commutative ring with\/ $1$ and let\/ $I,J\unlhd R$,
be ideals of\/ $R$. There exists a natural number\/ $N=N(\Phi)$
depending on\/ $\Phi$ alone, such that any commutator
$$ [x,y],\qquad x\in G(\Phi,R,I),\quad y\in E(\Phi,R,J) $$
\noindent
is a product of not more that\/ $N$ elementary generators
listed in Theorem~$\ref{The:3}$.
\end{The}

Quite remarkably, the bound $N$ in this theorem does not depend
either on the ring $R$, or on the choice of the ideals $I,J$.
The proof is based on the method of {\it universal localisation\/}
expressly developed by Stepanov \cite{stepanov10} to eliminate
any dependence on the dimension of $R$. Before that, even in the
absolute case all known bounds depended on dimension of the
ground ring, see \cite{SV10,portocesareo} and references there.

\par\smallskip
$\bullet$ Initially, one of the main motivations to consider relative
commutator subgroups occurred in the study of subgroups of classical
groups, normalised by a relative elementary subgroup. For the
general linear group, this problem was studied by John Wilson, Anthony
Bak, Leonid Vaserstein, Li Fuan and Liu Mulan, and the second author,
and in the context of unitary groups, by G\"unter Habdank, the third
author, and You Hong, one may find the bibliography in our survey
\cite{RN} and in the conference papers \cite{yoga,yoga2}. For
Chevalley groups this problem is still not solved, and it would be
a very interesting application.

\par\smallskip
$\bullet$ Another amazing fact, which follows from the {\it multiple\/}
relative commutator formula established for classical groups in
\cite{RZ12,RNZ3}, is that multiple commutators of relative subgroups
are reduced to double such commutators. More precisely, let $R$ be
a commutative ring with $1$ and let $I_i\unlhd R$, $i=1,\ldots,m$,
be ideals of $R$. Consider an arbitrary configuration of
brackets $[\![\ldots]\!]$ and assume that the outermost pairs
of brackets meet between positions $h$ and $h+1$. Then one has
\begin{multline*}
[\![E(\Phi,R,I_1),E(\Phi,R,I_2),\ldots,E(\Phi,R,I_m)]\!]=\\
[E(\Phi,R,I_1\ldots I_h),E(\Phi,R,I_{h+1}\ldots I_m)].
\end{multline*}
For Chevalley groups this result is not yet published, but it holds
as stated, and can be obtained either by imitating the methods of
\cite{RZ12,RNZ3}, or by Stepanov's universal localisation
\cite{stepanov10}.
\par\smallskip
$\bullet$ For {\it finite dimensional\/} rings, one has even stronger
{\it general\/} multiple commutator formulas, which are simultaneous
generalisations of the above Mason---Stothers formula, and nilpotency
of relative $\K_1$. In these formulas, elementary subgroups on the
left hand side replaced by the Chevalley groups themselves, provided
that $m>\dim(\Max(R))$. The definitive proofs are only published
for the general linear group, see \cite{HSVZmult}.
\par\smallskip
$\bullet$
Another very interesting occurence of the relative commutator groups
is relative splitting. In general, the intersection
$G(\Phi,R,I)\cap E(\Phi,R,J)$ of the principal congruence subgroup
modulo one ideal with the relative elementary subgroup modulo another
ideal, is very hard to describe. However, there are some instances,
when one could say more. Let us state the relative splitting principle
discovered by Himanee Apte and Alexei Stepanov, \cite{AS},
Lemma~\ref{Lem:4}. Let $I$ be a splitting ideal of an associative
ring $R$, $R=I\oplus A$, and let $J\unlhd R$ be an ideal of
$R$ generated by an ideal $K\unlhd A$. Then
$$ G(\Phi,R,I)\cap E(\Phi,R,J)=[E(\Phi,R,I),E(\Phi,R,J)]. $$
\par\smallskip
$\bullet$ Finally, let us mention that it seems that there is a
close connection between the relative commutator subgroups and
excision kernels. We refer to \cite{yoga2} for the references on
multiple relativisation and excision kernels, and some further
discussion.
\par\smallskip
In the present paper we considered only the usual relative groups,
defined in terms of a {\it single\/} ideal. Actually, for multiply
laced systems
relative subgroups should be parametrised not by the usual ideals,
but by {\it admissible pairs\/} $(A,B)$, as defined by Eiichi Abe
\cite{abe69}. In such a pair, the ideal $A$ parametrises elementary
generators for short roots $\Phi_s$, whereas the {\it additive
subgroup\/} $B$ parametrises elementary generators for long roots
$\Phi_l$. Moreover, $A_p\le B\le A$, where $p=2$ for
$\Phi=\B_l,\C_l,\F_4$ and $p=3$ for $\Phi=\G_2$. Here $A_p$ is
the ideal generated by $p\xi$ and $\xi^p$, for all $\xi\in A$.
In all non-symplectic cases $B$ is also an ideal of $R$, but in
the symplectic case it is a {\it Jordan ideal\/}. For classical
groups, admissible pairs are precisely the special case of form
ideals as defined the same year by Anthony Bak.
\par
In the non simply laced case, the genuine relative elementary
subgroups, which occur in the classification of normal subgroups
of $G(\Phi,R)$, are para\-met\-rised by admissible pairs, rather
than individual ideals, and are defined as follows:
$$ E(\Phi,R,A,B)={\left\langle x_{\alpha}(\xi),\alpha\in\Phi_s,\xi\in
A;\ x_{\beta}(\zeta),\beta\in\Phi_l,\zeta\in
B\right\rangle}^{E(\Phi,R)}. $$
\noindent
There is an analogue of Theorem~\ref{The:1} in this case, obtained by Michael
Stein \cite{stein2}, and explicitly stated by Eiichi Abe \cite{abe89}.

\begin{The}[{Stein---Abe}]\label{The:7}
Let $\Phi$ be a reduced irreducible root system of rank $\ge 2$
and let $I$ be an ideal of a commutative ring $R$. Then
as a group\/ $E(\Phi,R,A,B)$ is generated by the elements of
the form $z_{\alpha}(\xi,\eta)$, where\/ $\xi\in A$ for
$\alpha\in\Phi_s$ and $\xi\in B$ for $\alpha\in\Phi_l$,
while\/ $\eta\in R$.
\end{The}
Thus, the following problem naturally suggests itself.
\begin{Prob}
Generalise Theorems~$\ref{The:2}$ and $\ref{The:3}$ to the case
of the relative commutator subgroups
$\big[E(\Phi,R,A,B),E(\Phi,R,C,D)\big]$,
for two admissible pairs $(A,B)$ and $(C,D)$.
\end{Prob}

Actually, in the case, where $\Phi=\B_2,\G_2$ and the ring
has residue fields ${\Bbb F}_{\!2}$ of two elements,
the situation is even more complicated. Douglas Costa and
Gordon Keller studied the normal structure of $\Sp(4,R)$
and $G_2(R)$ in their remarkable papers \cite{costakeller1,costakeller2}
and discovered that in these cases relative subgroups should
be parametrised by {\it radices\/}, rather than by individual
ideals, or admissible pairs.

\begin{Prob}
For the cases $\Phi=\B_2$ and $\Phi=\G_2$ generalise
Theorems~$\ref{The:2}$ and $\ref{The:3}$ to the relative commutators
of elementary subgroups, defined in terms of radices.
\end{Prob}

We refer the interested reader to our papers cited above,
and  to our conference papers \cite{yoga,yoga2,portocesareo,RNZ4},
where one can find many further details and unsolved problems
concerning relative commutator subgroups.

{\bf Acknowledgement.}
We would like to thank Alexander Shchegolev who carefully read our original manuscript and
proposed several corrections.
The work of the second author is supported by the RFFI research projects 11-01-00756
(RGPU) and 12-01-00947 (POMI) and by the State Financed research task 6.38.74.2011
at the Saint Petersburg State University ``Structure theory and geometry of algebraic
groups and their applications in representation theory and algebraic K-theory''.


\end{document}